\documentclass[11 pt,a4paper,twoside]{article}
\usepackage{amsmath,amssymb,amsfonts,amsthm,epsfig,latexsym,makeidx,graphics,enumitem}
\usepackage[usenames,dvipsnames]{color}

 \def\NN{{\mathbb N}}  
\def\QQ{{\mathbb Q}} \def\RR{{\mathbb R}} \def\SS{{\mathbb S}} \def\TT{{\mathbb T}}

  \def\cG{{\cal G}}  \def\cS{{\cal S}}
    \def\cT{{\cal T}}
   \def\cO{{\cal O}}

\def\cF{{\cal F}}  \def\cL{{\cal L}}

\newtheorem{theorem}{{Theorem}}[section]
\newtheorem{proposition}[theorem]{{Proposition}}

\newtheorem{lemma}[theorem]{{Lemma}}

\newtheorem{corollary}[theorem]{{Corollary}}

\newtheorem{fact}[theorem]{{Fact}}

\newtheorem{problem}[theorem]{{Problem}}

\theoremstyle{definition}
\newtheorem{definition}[theorem]{{Definition}}

\theoremstyle{remark}
\newtheorem{remark}[theorem]{{Remark}}

\newtheorem{examples}[theorem]{{\it Examples}}

\title{A  spectral-like decomposition for transitive Anosov flows\\ in dimension three}
\author{F. Beguin, C. Bonatti and B. Yu
\thanks{The authors gratefully acknowledge financial support from the Agence Nationale de la Recherche (ANR project DynNonHyp) and the National Natural Science Foundation of China (NSFC 11471248) during the elaboration of this paper.}}

\date{\today}

\begin{document}

\maketitle

\hfill \emph{Dedicated to the memory of Dmitri V. Anosov (1936-2014)}

\begin{abstract}
Given a (transitive or non-transitive) Anosov vector field $X$ on a closed three-dimensional manifold~$M$, one may try to decompose $(M,X)$ by cutting $M$ along two-tori transverse to $X$. We prove that one can find a finite collection $\{T_1,\dots,T_n\}$ of pairwise disjoint, pairwise non-parallel incompressible tori transverse to $X$, such that the maximal invariant sets  $\Lambda_1,\dots,\Lambda_m$ of the connected components $V_1,\dots,V_m$ of $M-(T_1\cup\dots\cup T_n)$ satisfy the following properties:
\begin{itemize}[topsep=\parsep,itemsep=0cm]
\item[--] each $\Lambda_i$ is a compact invariant locally maximal transitive set for $X$;
\item[--] the collection $\{\Lambda_1,\dots,\Lambda_m\}$ is canonically attached to the pair $(M,X)$ (\emph{i.e.} it can be defined independently of the collection of tori $\{T_1,\dots,T_n\}$);
\item[--] the $\Lambda_i$'s are the smallest possible: for every (possibly infinite) collection $\{S_i\}_{i\in I}$ of tori transverse to $X$, the $\Lambda_i$'s are contained in  the maximal invariant set of $M-\cup_i S_i$.
\end{itemize}
To a certain extent, the sets $\Lambda_1,\dots,\Lambda_m$ are analogs (for Anosov vector field in dimension 3) of the basic pieces which appear in the spectral decomposition of a non-transitive axiom~A vector field. Then we discuss the uniqueness of such a decomposition: we prove that the pieces of the decomposition $V_1,\dots,V_m$, equipped with the restriction of the Anosov vector field $X$, are ``almost unique up to topological equivalence".
\end{abstract}

\section{Introduction}

Anosov flows (originally called $U$-system) are named after D.V. Anosov. In his celebrated paper \cite{Anosov1967}, Anosov proved that every Anosov flow is both structure stable and ergodic. Anosov flows are generalizations of the geodesic flow on closed Riemannian manifold with negative curvature, motivated by an important property of these geodesic flows: the whole underling manifold is a hyperbolic set for the flow. To be precise, for a closed  Riemannian manifold $M$, a nonsingular $C^r$ ($r\geq 1$) vector field $X$ on $M$ is called an \emph{Anosov vector field} (and the corresponding flow $(X^t)$ is called an \emph{Anosov flow}) if there exists  a $X$-invariant splitting $TM=E^s \oplus \mathbb{R} X \oplus E^u$ and some constants $C>0, \lambda>1$ such that: $\| DX_t (v) \| \leq C e^{-\lambda t} \| v \|$ for any $v\in E^s, t\geq 0$ and $\| DX_{-t} (v)  \| \leq C e^{-\lambda t} \| v \|$ for any $v\in E^u, t\geq 0$.

It is natural to develop a qualitative understanding of these flows. Many works have been done in this direction, although a complete classification seems to be absolutely out of reach, even in dimension~3. Let us cite a few papers:
\begin{itemize}
  \item Plante and Ghys classified Anosov flows on closed three-manifolds which are either torus bundles over the circle or circle bundles over a surface (\cite{Plante1981,Ghys1984}),
  \item Ghys classified Anosov flows on closed three-manifolds whose stable/unstable bundles are smooth  (\cite{Ghys1987}),
  \item Barbot and Fenley associated to each  Anosov flow $(X^t)$ on a closed three-manifold a pair of one-dimensional foliations $(\cG^s,\cG^u)$ on a subset of $\RR^2$, and explained how the dynamical properties of the flow $(X^t)$ translate as geometrical properties of the pair  $(\cG^s,\cG^u)$, and vice-versa (\cite{Fenley1994,Barbot1995});
  \item Barbot and Fenley gave a description of certain classes of (pseudo-)Anosov flows on certain toroidal manifolds (see e.g. \cite{BarbotFenley2013, BarbotFenley2014}).
\end{itemize}
Besides these efforts to classify some particular classes of Anosov flows in dimension~3, many new examples of Anosov have been discovered. For example:
\begin{itemize}
    \item Handel and Thurston have constructed an Anosov flow on a closed three-manifold which is neither a circle bundle over a surface, nor a surface bundle over the circle (\cite{HandelThurston1980});
     \item Goodman (\cite{Goodman1983}) has developed the technique used by Handel and Thurston. Her construction, called \emph{Dehn-Goodman surgery}, is a fundamental tool to construct various types of examples of Anosov flows;
    \item Franks and Williams have built an Anosov flow (on a closed three-manifold) which is not transitive (\cite{FranksWilliams1980});
    \item Bonatti and Langevin have constructed a transitive Anosov flow $X$ on a closed three-manifold $M$, such that there is a torus embedded in $M$ and transverse to $X$, but $X$ is not a suspension (\cite{BonattiLangevin1994}; see also~\cite{Barbot1998}).
 \end{itemize}
These examples show that the realm of Anosov flows in dimension three is much richer than what can be expected at first sight.
Both Franks-Williams' and Bonatti-Langevin's examples have been built by ``gluing hyperbolic plugs along their boundaries".
In~\cite{BeBoYu}, we have proved a technical result which allows to build many new examples of Anosov vector fields.
 Let us briefly recall this construction tool.

 \begin{definition}
 A \emph{hyperbolic plug} is a couple $(V,X)$ where $V$ is a compact three-dimensional manifold with boundary, and $X$ is a vector field on $V$,
 so that $X$ is transverse to the boundary of $V$, and so that the maximal invariant set $\Lambda:=\bigcap_{t\in\RR} X^t(V)$ is a saddle hyperbolic set for $X$.
 \end{definition}

 Consider such an orientable hyperbolic plug $(V,X)$, with maximal invariant set $\Lambda$. Denote by $\partial^{in} V$ (resp. $\partial^{out} V$)
 the union of the connected components of $\partial V$ along which the vector fields $X$ is point inwards (resp. outwards) $U$. Assume that there
  exists a diffeomorphism $\phi:\partial^{out} V\to\partial^{in} V$. Then one might consider the closed manifold $M_\phi:=V/\phi$ and the vector
  field $X_\phi$ induced by $X$ on $M_\phi$. In~\cite{BeBoYu}, we give some very general sufficient conditions for $X$ to be
   (topologically equivalent to) an Anosov vector field.
This allows us to produce many new examples of (transitive or non-transitive) Anosov vector fields in dimension three.

If $(M_\phi,X_\phi)$ has been constructed using the process described above, then one can recover the hyperbolic plug $(V,X)$ by cutting $M$ along a the projection of $\partial V$, \emph{i.e.} along finite collection of pairwise disjoint two-tori embedded in $M_\phi$ and transverse to $X_\phi$ \footnote{Recall that a closed connected surface $S$ which is transverse to an Anosov vector field $X$ in a closed orientable three-manifold is necessarily a two-torus, since the weak stable foliation of $X$ induces a non-singular foliation on $S$.}. More generally, any Anosov vector field on some closed three-manifold can be decomposed into hyperbolic plugs by cutting the manifold along tori that are transverse to the vector field. Let us formalize this:

\begin{definition}
\label{d.plug-decomposition}
Let $M$ be a closed orientable three-manifold, $X$ be an Anosov vector field on $M$, and $\cT=\{T_1,\dots,T_n\}$ be a finite collection of pairwise disjoint  tori in $M$ that are transverse to $X$. Choose $\epsilon>0$ small enough so that the tubular neighborhoods $X^{(-\epsilon,\epsilon)}(T_1),\dots,X^{(-\epsilon,\epsilon)}(T_n)$ of $T_1,\dots,T_n$ are pairwise disjoint, and denote by $V_1,\dots,V_m$ the connected components of $M\setminus \left(X^{(-\epsilon,\epsilon)}(T_1)\cup\dots\cup X^{(-\epsilon,\epsilon)}(T_n)\right)$. Then $(V_1,X_{|V_1}),\dots,(V_m,X_{|V_m})$ are hyperbolic plugs, and $(M,X)$ can be recovered (up to topological equivalence) by gluing these hyperbolic plugs together along their boundaries. We say that $V_1,\dots,V_m$ is a \emph{plug decomposition} of $(M,X)$ (more precisely, the plug associated of $(M,X)$ with the collection of tori $\cT$). 
\end{definition}

Similar to the JSJ decomposition of three-manifolds, one might hope to find a canonical plug decomposition of Anosov vector fields on three-manifold:

\begin{problem}
\label{p.decomposition}
Given an Anosov vector field $X$ on a closed orientable three-manifold $M$, can one find a ``canonical" plug decomposition of $(M,X)$~?\footnote{Of course, there exists a stupid canonical plug decomposition: indeed $(M,X)$ can be seen as the plug decomposition of $(M,X)$ associated to the empty collection of tori~! But one would like a plug decomposition which is not always trivial.}
\end{problem}

%

In order to discuss Problem~\ref{p.decomposition}, we are led to consider the collection of all the tori which are transverse to a given Anosov vector field:

\begin{definition}
\label{d.core}
Let $X$ be an Anosov vector field on a closed three-dimensional orientable manifold $M$. Let $\cT$ be the union of all the two-tori that are embedded in $M$ and transverse to $X$. We define the \emph{core} of $X$ to be the set $\mathrm{Core}(X):=M\setminus \cT$.
\end{definition}

In other words, a point $x\in M$ is in the core of $X$ if it impossible to find a two-torus $T$ embedded in $M$ and transverse to $X$, such that $x\in T$. Observe that, if a point $x$ belongs to a two-torus $T$ embedded in $M$, transverse to $X$, then the same is true for every point $y$ close to $x$, and for every point $y$ on the orbit of $x$. This shows that \emph{the core of an Anosov vector field $X$ is compact and invariant under the flow of $X$}.

\begin{examples}\label{examples1}
Let us describe the core of various examples of Anosov vector fields.
\begin{enumerate}
  \item  If $X$ is the suspension of Anosov diffeomorphism, then the core of $X$ is empty 
  (see corollary~\ref{c.characterization-suspension}).
  \item  If $M$ is the unit tangent bundle of a closed hyperbolic surface and $X$ is the geodesic flow on $M$, 
  then there does not exist any torus embedded in $M$ and transverse to $X$  (indeed $X$ which is topologically equivalent to $-X$ by a topological equivalence which is isotopic to the identity~; one can easily prove that transitive Anosov vector fields bearing this property do not admit any transverse torus, see \emph{e.g.} \cite[proof of Proposition 4.9]{Barbot1995}). Therefore the core of $X$ is the whole manifold $M$.
  \item If $M$ is atoroidal, then, for every Anosov vector field $X$ on $M$, the core of $X$ is the whole manifold $M$. The first examples of Anosov vector fields on closed atoroidal three-dimensional manifolds were constructed by Goodman (\cite{Goodman1983}).
  \item The core of the Bonatti-Langevin example (\cite{BonattiLangevin1994}) is a single periodic orbit $\gamma$. Let us explain why. The Bonatti-Langevin manifold $M$ is a graph manifold with one Seifert piece $M-T$ obtained by cutting $M$ along a JSJ torus $T$. The maximal invariant set  of $M-T$ for the Bonatti-Langevin vector field $X$ is a periodic orbit $\gamma$. Therefore the core of $X$ is either empty or reduced to $\gamma$. But $\gamma$ is isotopic to the regular fiber of $M-T$, and the algebraic intersection number of the regular fiber of $M-T$ with every incompressible torus equals $0$. Since every torus transverse to an Anosov vector field is incompressible (\cite{Brunella1993,Fenley1995}), it follows that $\gamma$ does not intersect any torus transverse to $X$. Hence the core of $X$ is the periodic of orbit $\gamma$.
  \item  If $X$ is a non-transitive Anosov flow, the existence of Lyapunov functions implies that the core of $(M,X)$ is contained in the non-wandering set of $X$. In particular, in the case the non-transitive Anosov flow on a 3-manifold $M$ by gluing two figure eight  knot complement spaces constructed  by Franks and Williams, the core of $X$ exactly is the non-wandering set of $X$ since the gluing torus is the unique incompressible torus in $M$.
  \item  In~\cite{BeBoYu}, we have constructed many examples of transitive Anosov flows with complicated cores (typically, the core of such examples will be a transversally cantorian set). We will explain more about them later (see example~\ref{e.fine-decomposition}).
\end{enumerate}
\end{examples}


Let us recall that a \emph{basic set} (for a vector field $X$) is a compact transitive locally maximal hyperbolic set. Smale proved that every locally maximal hyperbolic set has only finitely many chain-recurrence classes, each of which is a basic set. As a consequence, every locally maximal hyperbolic set can be decomposed, in a unique way, as a 
finite union of pairwise disjoint basic sets. We shall prove the following theorem:

%

\begin{theorem}
\label{t.main}
Let $X$ be an Anosov vector field on a closed connected orientable 3-manifold $M$. Assume that $X$ is not topologically equivalent to a suspension.
\begin{enumerate}
\item The core of $X$ is a locally maximal set. It can be decomposed, in a unique way, as a finite union of pairwise disjoint basic sets: $\mathrm{Core}(X)=\Lambda_1\sqcup\dots\sqcup\Lambda_m$.
\item There exists a plug decomposition $V_1,\dots,V_m$ of $(M,X)$ such that the basic set $\Lambda_i$ is the maximal invariant set of $V_i$ for every $i\in\{1,\dots,m\}$.
\end{enumerate}
\end{theorem}

\begin{definition}
The basic sets $\Lambda_1,\dots,\Lambda_m$ defined by item~1 of Theorem~\ref{t.main} will be called the \emph{core basic sets} of $X$.  A plug decomposition of $(M,X)$  satisfying item~2 of Theorem~\ref{t.main} (\emph{i.e.} a plug decomposition $V_1,\dots,V_m$ such that the maximal invariant set of $V_i$ is the core basic set $\Lambda_i$ for every $i$) will be called a \emph{fine plug decomposition}.
\end{definition}

The adjective \emph{fine}  is justified by the following observation.

\begin{remark}
Let $V_1,\dots,V_m$ be a fine plug decomposition of $(M,X)$, and $W_1,\dots,W_p$ be an arbitrary plug decomposition of $(M,X)$. By definition of a fine plug decomposition, the maximal invariant set of $V_i$ is the core basic set $\Lambda_i$ defined by item~1 of Theorem~\ref{t.main}. Denote by $L_i$ the maximal invariant set of $W_j$. Then $L_1\sqcup\dots\sqcup L_p$ is the maximal invariant set the complement of a finite union of tori in that are transverse to $X$. Hence, the core of $X$ is contained in $L_1\sqcup\dots\sqcup L_p$. Hence, for every $i$, the basic set $\Lambda_i$ is contained in $L_1\sqcup\dots\sqcup L_p$. And since $\Lambda_i$ is connected (it is transitive), it must be contained in $L_{j_i}$ for some $j_i\in\{1,\dots,p\}$. In other words, \emph{the maximal invariant sets of the elements of a fine plug decomposition of $(M,X)$  are contained in the maximal invariants sets of the elements of any other plug decomposition of $(M,X)$}.
\end{remark}


\begin{examples}\label{e.fine-decomposition}
Let us describe the outcome of theorem~\ref{t.main} in various situations.
\begin{enumerate}
\item If there does not exist any two-torus transverse to $X$ (\emph{e.g.} if $M$ is atoroidal, or if $M$ is 
the unit tangent bundle of an hyperbolic compact surface and $X$ is the geodesic flow on $M$), 
then $\mathrm{Core}(X)=M$ is a basic set by itself, and there is no torus embedded in $M$ and transverse to $X$. With the notations of theorem~\ref{t.main}: $m=1$ and $n=0$. Of course, in this situation, theorem~\ref{t.main} is completely trivial and useless.
\item If $(M,X)$ is the example constructed by Bonatti and Langevin in~\cite{BonattiLangevin1994}, then $\mathrm{Core}(X)$ is a single periodic orbit, and there is only one torus embedded in $M$ and transverse to $X$ up to isotopy along the orbits of $X$. With the notations of theorem~\ref{t.main}: $m=n=1$.
\item More generally, for the examples of Anosov flows on Seifert or graph manifolds studied by Barbot and 
Fenley in~\cite{BarbotFenley2013}, $\mathrm{Core}(X)$ is made of a finite number of saddle periodic orbits, 
and the basic sets $\Lambda_1,\dots,\Lambda_m$ are precisely these periodic orbits. 
This can be proved using the same arguments as in Examples~\ref{examples1} for the Bonatti-Langevin example.
\item If  $(M,X)$ is the example constructed by Franks and Williams in~\cite{FranksWilliams1980} where $M$ can be obtained by gluing two figure eight  knot complement spaces.  $\mathrm{Core}(X)$ is made of two basic sets $\Lambda_1$ and $\Lambda_2$. In this case, $m=2$ and $n=1$.
\item Our paper~\cite{BeBoYu} provides many examples where the basic sets $\Lambda_1,\dots,\Lambda_m$ are non-trivial saddle hyperbolic sets (\emph{i.e.} hyperbolic sets that are not reduced to periodic orbits). The simplest way to construct such examples is the following. Take a transitive Anosov vector field $X_0$ on a manifold $M_0$, and pick up two periodic orbits $O$ and $O'$ of this vector field. Make a DA-bifurcation near $O$ is the stable direction and a DA bifurcation in the unstable direction near $O'$. This creates a new vector field $X_1$ with an attracting periodic orbit and a repelling periodic orbit. Let $M_1$ be the manifold with boundary obtained by excising a small tubular neighborhood of each of these two periodic orbits. The boundary of $M_1$ is made of two tori $T$ and $T'$, that are transverse to $X_1$. By gluing $T$ on $T'$, one gets a closed manifold $M$, endowed with a vector field $X$ induced by $X_{1|M_1}$.  In~\cite{BeBoYu}, we have proved that it is possible to perform this gluing in such a way that $X$ is a transitive Anosov vector field. If there does not exist any torus embedded in the manifold $M_0$ transverse to the vector field $X_0$, then the core of the new vector field $X$ is a transitive non-trivial saddle hyperbolic set (and therefore, for this vector field, $m=n=1$).
 \end{enumerate}
\end{examples}

To what extend does Theorem~\ref{t.main} solve problem~\ref{p.decomposition}~? Theorem~\ref{t.main}  states the existence of a particular plug decomposition, which we call a \emph{fine plug decomposition}. Is this fine plug decomposition unique~? The answer to this question depends on the equivalence relation we put on the set of plug decompositions. Several natural equivalence relations can be proposed.

\begin{definition}
Let $X$ be an Anosov vector field on a closed orientable three-manifold $M$. Two plug decompositions $V_1,\dots,V_p$ and $W_p,\dots,W_p$ of $(M,X)$ are said to be:
\begin{enumerate}
\item \emph{residually equivalent} if, for every $i$, the plugs $(V_i,X_{|V_i})$ and $(W_{\sigma(i)},X{|W_{\sigma(i)}})$ have the same maximal invariant set,
\item \emph{piecewise topologically equivalent} if, for every $i$, there exists a homeomorphism $h_i:V_i\to W_{\sigma(i)}$ mapping each oriented orbit of $X_{V_i}$ on an oriented orbit of $X_{|W_{\sigma(i)}}$,
\item \emph{globally topologically equivalent} if there exists a homeomorphism $h:M\to M$ mapping each oriented orbit of $X$ on an oriented orbit of $X$, and such that $h(V_i)=W_{\sigma(i)}$ for every $i$,
\item \emph{flow isotopy equivalent} if there exists a function $\theta:M\to \RR$ such that the homeomorphism $h:M\to M$ defined by $h(x):=X^{\theta(x)}(x)$ maps $V_i$ on $W_{\sigma(i)}$ for every $i$,
\end{enumerate}
where $\sigma$ is a permutation of $\{1,\dots,n\}$.
\end{definition}

Clearly, flow isotopy equivalence implies global topological equivalence, which implies piecewise topological equivalence, which implies residual equivalence. The following fact is a straightforward consequence of the definitions:

\begin{fact}
Two fine plug decompositions of the same Anosov vector field are residually equivalent.
\end{fact}

We will prove the following:

\begin{theorem}
\label{t.finitely-many}
An Anosov vector field on a closed orientable three-manifold admits at most finitely many fine plug decompositions up to piecewise topological equivalence.
\end{theorem}

Moreover, we will describe some particular case where a fine plug decomposition is necessarily unique up to piecewise topological equivalence; see Proposition~\ref{p.filling-lamination-case}. On the other hand, we will give an example showing that fine plug decompositions are not unique up to flow isotopy equivalence:

\begin{proposition}
\label{p.infinitely-many}
There exists an Anosov vector field $Z$ on a closed orientable three-manifolds $M$, such that $(M,Z)$ admits infinitely many fine plug decompositions which are pairwise not flow isotopy equivalent.
 \end{proposition}

 We do not know if there exists Anosov vector fields which admit infinitely many fine plug decompositions that are pairwise not globally topologically equivalent. As a summary, the plug decomposition provided by Theorem~\ref{t.main} is:
\begin{itemize}
\item canonical from the viewpoint of residual equivalence,
\item ``almost canonical" from the viewpoint of piecewise topological equivalence,
\item not canonical  in general from the viewpoint of flow isotopy equivalence.
\end{itemize}
Roughly speaking, this means that~: the hyperbolic plugs provided by Theorem~\ref{t.main} are almost unique (finitely many possibilities) when we consider them intrinsically, but the positions of these plugs with respect of the orbits of $X$ in the closed manifold $M$ are far from being unique (infinitely many possibilities).

\bigskip

Theorem~\ref{t.main} is obviously reminiscent of Smale's decomposition theorems for axiom A vector fields. Recall that non-singular $C^1$ vector field $X$ on a three-manifold $M$ is said to be \emph{axiom~A} if its non-wandering set $\Omega(X)$ is hyperbolic and coincides with closure of the periodic points of $X$ (Anosov vector fields are examples of axiom~A vector fields). Under this hypothesis, $\Omega(X)$ can be decomposed, in a unique way, as the union of a finite number of pairwise disjoint basic sets $K_1,\dots,K_r$: this is the so-called \emph{spectral decomposition theorem}, and the sets $K_1,\dots,K_M$ are called the \emph{basic pieces} of $X$, see~\cite{Smale1967}. Moreover, using Lyapunov function, one can find a plug decomposition $(W_1,X),\dots,(W_r,X)$ of $(M,X)$ so that the maximal invariant set of $W_i$ is precisely the basic piece $K_i$. Moreover, in the particular case where $X$ is an Anosov vector field, the surfaces $S_1,\dots,S_n$ are necessarily tori, and Brunella has proved these tori $S_1,\dots,S_n$ are incompressible and pairwise non-parallel (\cite{Brunella1993}). So Theorem~\ref{t.main} is somehow parallel to Smale's decomposition of axiom A vector fields, where:
\begin{itemize}
\item the core of $X$ play the same role as the non-wandering set of an axiom A vector field;
\item the core basic sets $\Lambda_1,\dots,\Lambda_m$ provided by Theorem~\ref{t.main} play the same role as the basic pieces $K_1,\dots,K_r$ which appear in \emph{Smale's spectral decomposition} of the non-wandering set;
\item the fine plug decomposition $(V_1,X),\dots,(V_m,X)$ provided by Theorem~\ref{t.main} play the same role as the plug decomposition $(W_1,X),\dots,(W_r,X)$ defined above by mean of a Lyapunov function.
\end{itemize}
This is the reason we speak of ``spectral-like decomposition for Anosov vector fields". Notice nevertheless that the core basic sets $\Lambda_1,\dots,\Lambda_m$ provided by Theorem~\ref{t.main} are in general much smaller than the basic pieces of $X$ provided by Smale's decomposition theorem (see  e.g. by example~2, 3 and 5 above), and  a single basic piece of $X$ might contain several core basic sets (see e.g.~\cite[Theorem 1.5]{BeBoYu}).
Also notice that the proof of theorem~\ref{t.main} is completely different from those of Smale's spectral decomposition theorem for non-transitive axiom~A vector fields. Indeed, the starting point of the proof Smale's spectral decomposition theorem is the fact that the non-wandering set of an axiom A vector field is locally maximal. But we do not know \emph{a priori} that the core of an Anosov flow is a locally maximal set. Roughly speaking\footnote{Actually, the proof of theorem~\ref{t.main} is slightly more intricate than described below.}~:
\begin{itemize}[topsep=\parsep,itemsep=0cm]
\item[--] in the case of non-transitive axiom A flows, one first observes that the non-wandering set is locally maximal, hence gets a spectral decomposition, and therefore (using Lyapunov theory) gets some surfaces  $S_1,\dots,S_n$ which separates the basic pieces from each others,
\item[--] to obtain theorem~\ref{t.main}, we first have to prove that there exists a finite family of pairwise disjoint, pairwise non-parallel two-tori $\{T_1,\dots,T_n\}$ so that $\mathrm{Core}(X)$ is the maximal invariant set of $M-(T_1\cup\dots\cup T_i)$. This will imply $\mathrm{Core}(X)$ is a locally maximal set. Once we will know that, the decomposition in core basic sets will  follow by Smale's classical arguments.
\end{itemize}

\section{Existence of a fine plug decomposition}

The purpose of this section is to prove Theorem~\ref{t.main}.
All along this section, we consider a closed connected orientable three-dimensional manifold $M$, and an Anosov vector field $X$ on $M$. In the sequel, the phrase ``\textit{a transverse torus}" will always mean ``\textit{a two-dimensional torus embedded in $M$ and transverse to the vector field $X$}". The proof of theorem~\ref{t.main} uses five ingredients~:
\begin{enumerate}[topsep=\parsep,itemsep=0cm]
\item a cohomological argument which allows to find a \emph{finite} collection of transverse tori which intersect every periodic orbit in $M- \mathrm{Core}(X)$ (proposition~\ref{p.finite});
\item a desingularisation procedure which allow to replace any finite collection of transverse tori by a collection of \emph{pairwise disjoint} transverse tori $\cT'$ intersecting exactly the same orbits of $X$ as $\cT$ (corollary~\ref{c.disjoint});
\item the classical spectral decomposition theorem which tells us that $\Lambda_{\cT'}$ contains only finitely many chain recurrence classes, and that the periodic orbits are dense in each of these chain recurrence classes (theorem~\ref{t.basic-sets-decomposition});
\item Lyapunov theory which provides us with some transverse tori separating the basics sets of $\Lambda_{\cT'}$ (theorem~\ref{t.Lyapunov});
\item a lemma of Brunella which tells that two parallel transverse tori intersect the same orbits of $X$ (lemma~\ref{l.non-parallel}).
\end{enumerate}

We now proceed to the proof:

\begin{proposition}
\label{p.finite}
There exists a finite collection of transverse tori $\cT=\{T_1,\dots,T_n\}$ such that $T_1\cup\dots\cup T_n$ intersects every periodic orbit which is not in the core of $X$.
\end{proposition}

\begin{proof}[Proof of proposition~\ref{p.finite}]
We argue by contradiction: we assume that the conclusion of proposition~\ref{p.finite} does not hold. Using this assumption, a immediate induction allows us to construct a sequence $(T_k)_{k\geq 1}$ of transverse tori and a sequence $(\gamma_k)_{k\geq 2}$ of periodic orbits of $X$ with the following property~:
\begin{itemize}[topsep=\parsep,itemsep=0cm]
\item[(P)] For every $k\geq 2$, the orbit $\gamma_k$ is disjoint from the torus $T_i$ for $i=1,\dots,k-1$, but does intersect the torus $T_k$.
\end{itemize}
We will show that property (P) yields a contradiction. For this purpose, we will use the cohomology $H^*(M)=H^*(M,\QQ)$ of the manifold $M$. Consider an orbit $\gamma$ of $X$, and a two-torus $T$ embedded in $M$, transverse to the vector field $X$. By Poincare duality, $\gamma$ corresponds to a class $[\gamma] \in H^2(M)$, and  $T$ corresponds to a class $[T] \in H^1 (M)$. A standard theorem in (co-)homology theory tell us that the cup product of $[\gamma]$ and $[T]$ satisfies $[\gamma] \vee [T] = \mathrm{int}(\gamma,T) [\cdot]$, where $\mathrm{int}(\gamma,T)$ is the algebraic intersection number of $\gamma$ and $T$, and $[\cdot]$ is the Poincar\'e dual of the homology class of a point. Now, observe that, since $\gamma$ is an orbit of $X$, and  $T$ is transverse to $X$, the algebraic intersection number  $\mathrm{int}(\gamma,T)$ is nothing but the cardinal of $\gamma\cap T$. Therefore, $[\gamma] \vee [T]=0$ if and only if $\gamma\cap T=\emptyset$. This allows us to translate property~(P) as follows~:
\begin{itemize}[topsep=\parsep,itemsep=0cm]
\item[(P')] for every $k\geq 2$, one has $[\gamma_{k}]\vee [T_i]=0$ for $i=1,\dots,k-1$, but $[\gamma_k]\vee [T_k]\neq 0$.
\end{itemize}
In particular, for every $k\geq 0$, the class $[T_k]$ is not in the linear subspace of $H^1(M)$ spanned by $[T_1],...,[T_{k-1}]$. Therefore, $H^1(M)$ must be infinite dimensional. But this is impossible since $M$ is a closed manifold. This completes the proof of proposition~\ref{p.finite}.
\end{proof}

We will now explain how to replace the finite collection of transverse tori provided by proposition~\ref{p.finite} by a finite collection of \emph{pairwise disjoint} transverse tori~:

\begin{proposition}
\label{p.disjoint}
Given any finite collection $\cS=\{S_1,\dots,S_k\}$ of transverse tori such that $S_1,\dots,S_{k-1}$ are pairwise disjoint, one can find another finite collection $\widehat S=\{\widehat S_1,\dots,\widehat S_\ell\}$ of transverse tori such that:
\begin{itemize}[topsep=\parsep,itemsep=0cm]
\item $\widehat S_1,\dots,\widehat S_\ell$ are pairwise disjoint;
\item $\widehat S:=\widehat S_1\cup\dots\cup\widehat S_\ell$ intersects exactly the same orbits of $X$ as $S:=S_1\cup \dots\cup S_k$.
\end{itemize}
\end{proposition}

The proof of this proposition consists in performing some simple surgeries to eliminate on the immersed surface $S=S_1\cup \dots\cup S_k$.

\begin{proof}[Proof of proposition~\ref{p.disjoint}]
 Let us first observe that a small perturbation of the embedding of $S_k$ does not change the set of orbits of $X$ that are intersected by this torus. Therefore, we can assume that $S_k$ is transverse to $S_1,\dots,S_{k-1}$. It follows that $S_k\cap (S_1\cup\dots\cup S_{k-1})$ is a finite collection of pairwise disjoint simple closed curves.

Let $C$ be one of these curves; $C$ is a connected component component of $S_i\cap S_k$ for some $i\in\{1,\dots,k-1\}$. Since $X$ is transverse to both $S_i$ and $S_n$, we can find a neighborhood $U$ of $C$ in $M$, and a diffeomorphism $h:U\to \SS^1\times [-1,1]\times [-1,1]$ such that:
\begin{itemize}[topsep=\parsep,itemsep=0cm]
\item the surface $h(S_i)$ (resp. $h(S_k)$) is a graph over the ``horizontal" annulus $\SS^1\times [-1,1]\times \{0\}$;
\item the vector field $h_*X$ is ``vertical": the flow lines of $h_*X$ are the vertical lines $\{\star\}\times\{\star\}\times [-1,1]$
\end{itemize}
(see figure~\ref{f.surgery}, left). Then we can find two smooth graphs $A,A'\subset U$ over the annulus $\SS^1\times [-1,1]\times \{0\}$ such that:
\begin{itemize}[topsep=\parsep,itemsep=0cm]
\item $A$ is disjoint from $A'$ (say $A'$ is strictly below $A$);
\item $A\cup A'$ coincides with $h(S_i)\cup h(S_k)$ on a neighborhood of the boundary of $\SS^1\times [-1,1]\times \{0\}$.
\end{itemize}
(see the figure below rightwards). Note that $h^{-1}(A)$ (resp. $h^{-1}(A')$) is transverse to $X$, since $A$ (resp. $A'$) is a graph over the horizontal annulus $\SS^1\times [-1,1]\times \{0\}$ and $X$ is vertical. Also note that $h^{-1}(A\cup h^{-1}(A')$ intersects exactly the same orbits of $X$ as $(S_i\cup S_k)\cap U$ (because every horizontal graph in over the horizontal annulus $\SS^1\times [-1,1]\times \{0\}$ intersects every vertical line $\{\star\}\times\{\star\}\times [-1,1]$ in $\SS^1\times [-1,1]\times [-1,1]$).

We do a surgery on $S=S_1\cup\dots\cup S_k$: we replace  $(S_i\cup S_k)\cap U$ by  $h^{-1}(A\cup h^{-1}(A')$. Actually, we do this surgery not only for the curve $C$, but for each connected component of $S_k\cap (S_1\cup\dots\cup S_{k-1})$. This surgery changes $S=S_1\cup\dots\dots S_k$ into a (not necessarily connected) closed surface $\widehat S$ which is embedded in $M$ and transverse to $X$. This surface $\widehat S$ intersects exactly the same orbits of $X$ as $S$. The connected component of $\widehat S$ are transverse tori, since every connected closed surface transverse to an Anosov vector field is a torus. We set $\widehat{\cS}=\{\widehat S_1,\dots,\widehat S_{\hat\ell}\}$ to be the collection of the connected components of $\widehat S$.
\end{proof}

\begin{figure}[h]
\centerline{\includegraphics[totalheight=5.5cm]{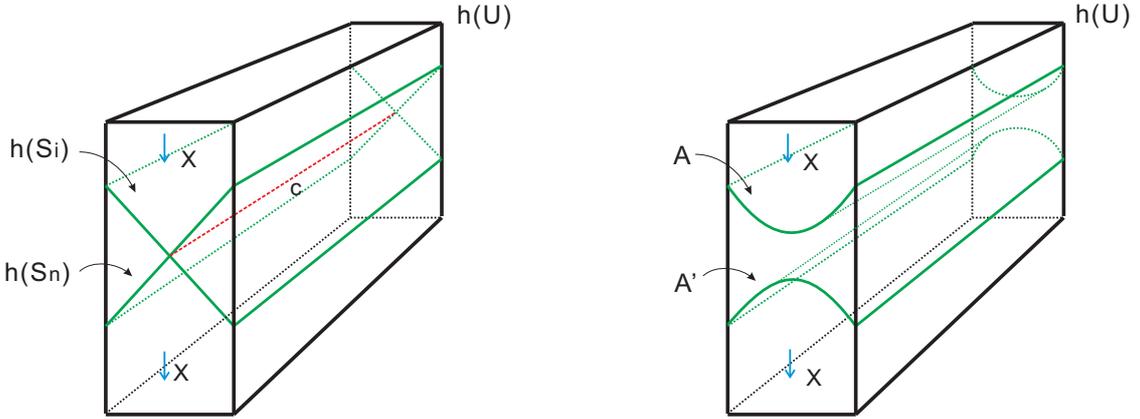}}
\caption{\label{f.surgery}Removing intersection curves in a union of transverse tori}
\end{figure}

\begin{corollary}
\label{c.disjoint}
Given any finite collection of transverse tori $\cT$, there exists another finite collection of transverse tori $\cT'$ such that:
\begin{itemize}[topsep=\parsep,itemsep=0cm]
\item the elements of $\cT'$ are pairwise disjoint;
\item the union of the elements of $\cT'$ intersects exactly the same orbits of $X$ as the union of the elements of $\cT$.
\end{itemize}
\end{corollary}

\begin{proof}
Denote by $T_1,\dots,T_n$ the elements of $\cT$. The proof consists in applying $n-1$ times proposition~\ref{p.disjoint}. First, we consider the collection of transverse tori $\cS^1:=\{T_1,T_2\}$. Proposition~\ref{p.disjoint} provides a collection of pairwise disjoint transverse tori $\widehat{\cS^1}$, such that the union of the elements of $\widehat{\cS^1}$ intersect exactly the same orbits of $X$ as the union of the elements of $\cS^1$, \emph{i.e.} the same orbits of $X$ as $T_1\cup T_2$. Then we consider the collection of transverse tori $\cS^2:=\widehat{\cS^1}\cup\{T_3\}$. Proposition~\ref{p.disjoint} provides a collection of pairwise disjoint tori $\widehat{\cS^2}$, such that the union of the elements of $\widehat{\cS^2}$ intersect the same orbits of $X$ as the union of the elements of $\cS^2$, \emph{i.e.} the same  orbits of $X$ as $T_1\cup T_2\cup T_3$. \textit{Etc.}
\end{proof}

Let us now recall Conley's fundamental theorem of dynamical systems (see \emph{e.g.} \cite[Theorem 3.14]{Shub1978}):

\begin{theorem}[Conley]
\label{t.Lyapunov}
Let $N$ be a compact manifold with boundary, and $Y$ be a vector field on $N$, transverse to $\partial N$. Let $\partial^{in}N$ (resp. $\partial^{out} N$) be the part of $\partial N$ where $Y$ is pointing inwards (resp. outwards) $N$. Then one can find a smooth function $f:N\to [-1,1]$ such that:
\begin{enumerate}[topsep=\parsep,itemsep=0cm]
\item $\partial^{in} N=f^{-1}(\{1\})$ and $\partial^{out} N=f^{-1}(\{1\})$;
\item every chain-recurrence class of $Y$  is contained in a (singular) level set of $f$;
\item any two distinct chain-recurrence classes of $Y$ are contained in distinct level sets of $f$;
\item $df_x.Y<0$ for every point $x$ which is not in the chain-recurrent set of $Y$ (in particular, $f$ is decreasing along any orbit of $Y$ which is not in the chain recurrent set of $Y$).
\end{enumerate}
\end{theorem}

The function $f$ is called \emph{a complete smooth Lyapunov function for the vector field $Y$}. We also recall Smale's spectral decomposition theorem for locally maximal hyperbolic sets (see \emph{e.g.}~\cite{Shub1978}):

\begin{theorem}[Smale]
\label{t.basic-sets-decomposition}
Let $N$ be a compact manifold (with or without boundary), $Y$ be a vector field on $N$ transverse to the boundary of $N$, and $\Lambda$ be a compact invariant locally maximal hyperbolic set for $Y$. Then $Y_{|\Lambda}$ has finitely many chain-recurrent classes, called \emph{basic pieces}. Each basic piece is a \emph{basic set}, 
\emph{i.e} is transitive, locally maximal, contains a dense set of periodic orbits.
\end{theorem}

To complete the proof of theorem~\ref{t.main}, we shall need the following lemma, which was proved by Brunella in~\cite{Brunella1993}:

\begin{lemma}[Brunella]
\label{l.non-parallel} Let $X$ be an Anosov vector field on a $3$-manifold $M$. 
If $N$ is a sub-manifold with boundary of $M$, which is diffeomorphic to $\TT^2\times [0,1]$, and whose boundary is transverse to $X$, then the maximal invariant set $\bigcap_{t\in\RR} X^t(N)$ is empty.
\end{lemma}

The proof of this lemma is quite short, but relies on rather sophisticated tools (a version of Novikov's theorem due to Plante). The converse of the above lemma is also true (and more easier to prove):

\begin{lemma}
\label{l.parallel} Let $X$ be an Anosov vector field on an orientable $3$-manifold $M$.
Let $N$ be a non-empty connected sub-manifold with boundary of $M$, so that the boundary of $N$ is transverse to $X$. 
If the maximal invariant set $\Lambda:=\bigcap_{t} X^t(N)$ is empty, then $N$ is diffeomorphic to $\TT^2\times [0,1]$, and $X_{|N}$ is topologically equivalent to vertical vector field $\frac{\partial}{\partial t}$ on $\TT^2\times [0,1]$.
\end{lemma}

\begin{proof}
Recall that we denote by $\partial^{in} N$ (resp. $\partial^{out} N$) the union of the connected components of $\partial N$ where $X$ is pointing inwards (resp. outwards) $N$. First observe that $\partial^{in} N$ is non-empty: indeed, if both $\partial^{in} N$ was empty, them the backward orbit of a point in $N$ could not exit $N$, and therefore would accumulate on a non-empty maximal invariant set, contradicting our assumption. Now, for every $x\in \partial^{in} N$, the forward orbit of $x$ eventually exits $N$ by cutting $\partial^{out} N$ (otherwise it would accumulate on a non-empty maximal invariant set) . Moreover, since $\partial^{in} N$ and $\partial^{out} N$ are transverse to $X$, the exit time $\tau(x)$ depends continuously on $x$. Hence, the map $(x,t)\mapsto X^{\frac{t}{\tau(x)}}(x)$ defines a diffeomorphism between  $\partial^{in} N\times [0,1]$ and $N$. Since $N$ is connected, it follows that $\partial^{in} N$ must also be connected. Therefore $\partial^{in} N$ is a two-torus, and $N$ is diffeomorphic to $\TT^2\times [0,1]$, and $X_{|N}$ is topologically equivalent to the vertical vector field $\frac{\partial}{\partial t}$ on $\TT^2\times [0,1]$.
\end{proof}

The preceeding result allow to characterize Anosov flows with empty core:

\begin{corollary}
\label{c.characterization-suspension}
If $\mathrm{Core}(X)$ is empty, then $X$ is topologically equivalent to the suspension of an Anosov diffeomorphism of $\TT^2$.
\end{corollary}

\begin{proof}
According to Proposition~\ref{p.finite} and Corollary~\ref{c.disjoint}, we can find a finite collection of transverse tori $\cT=\{T_1,\dots,T_n\}$ such that $T_1\cup\dots\cup T_n$ intersects every periodic orbit of $X$. Let us denote $V_1,\dots,V_m$ the plug decomposition of $M$ associated with $\cT$ (definition~\ref{d.plug-decomposition}). By construction, for every $i$, the maximal invariant set of $V_i$ does not contain any periodic orbit of $X$. By Smale's Theorem~\ref{t.basic-sets-decomposition}, it follows that the maximal invariant set of $V_i$ is empty. And by lemma~\ref{l.non-parallel}, it follows that $V_i$ is diffeomorphic to $\TT^2\times [0,1]$ and  and $X_{|V_i}$ is topologically equivalent to the vertical vector field $\frac{\partial}{\partial t}$ on $\TT^2\times [0,1]$. One easily deduces that the $V_i$'s are glued ``cyclically" ($n=m$ and, up to permutation of the indices, the torus $T_i$ separates $V_i$ from $V_{(i+1)\mathrm{mod} n}$) and every orbit of $X$ intersects each of the tori $T_1,\dots,T_n$. In particular, there is a transverse torus intersecting every orbit of $X$. Hence, $X$ is topologically equivalent to a suspension.
\end{proof}

We can now complete the proof of theorem~\ref{t.main}:

\begin{proof}[Proof of theorem~\ref{t.main}]
For every finite collection $\cS=\{S_1,\dots,S_n\}$ of pairwise disjoint transverse tori, 
we will denote by $M_{\cS}$ the compact manifold with boundary obtained by cutting $M$ along $\cS$. 
More precisely, $M_{\cS}=M\setminus \left(X^{(-\epsilon,\epsilon)}(S_1)\cup\dots\cup X^{(-\epsilon,\epsilon)}(S_n)\right)$, 
where $\epsilon>0$ is chosen small enough so that the tubular neighborhoods 
$X^{(-\epsilon,\epsilon)}(S_1),\dots,X^{(-\epsilon,\epsilon)}(S_n)$ are pairwise disjoint. 
Obviously, $M_{\cS}$ does not depend on the choice of $\epsilon$ up to isotopy along the orbits of $X$. 
Note that the plug decomposition of $M$ associated with $\cS$ (definition~\ref{d.plug-decomposition}) is 
nothing but the collection of the connected component of $M_\cS$. 
We will denote by $\Lambda_\cS$ the maximal invariant set of $M_{\cS}$. 

By assumption, $X$ is a not topologically equivalent to a suspension. 
According to corollary~\ref{c.characterization-suspension}, 
it follows that $\mathrm{Core}(X)$ is non-empty.  
As a consequence, for every collection $\cS$ pairwise disjoint transverse tori, 
the maximal invariant set $\Lambda_\cS$ is non-empty: indeed $\Lambda_\cS$ contains $\mathrm{Core}(X)$ which is not empty.

\bigskip

Proposition~\ref{p.finite} together with corollary~\ref{c.disjoint} provide 
a finite collection of pairwise disjoint transverse tori $\cT:=\{T_1,\dots,T_p\}$ such that 
$T_1\cup\dots\cup T_p$ intersects every periodic orbit in $M-\mathrm{Core}(X)$. 
We consider the compact manifold with boundary $M_\cT$ obtained by cutting $M$ along $\cT$, 
and the maximal invariant set $\Lambda_\cT$ of $M_\cT$. The set $\Lambda_{\cT}$ is a compact, invariant, locally maximal, 
and hyperbolic for $X$. By theorem~\ref{t.basic-sets-decomposition}, $X_{|M_{\cT}}$ has only finitely many chain recurrence classes 
$\Lambda_1,\dots,\Lambda_m$, and the periodic orbits of $X$ are dense in each $\Lambda_i$.

Theorem~\ref{t.Lyapunov} provides us with a complete smooth Lyapunov function $f:M_{\cT}\to [-1,1]$ for $X_{|M_{\cT}}$. 
So, for every $i$, there exists a real number $-1<a_i<1$ such that $\Lambda_i \subset f^{-1}(a_i)$, and $a_i\neq a_j$ for $i\neq j$. 
Changing the indexation of the $\Lambda_i$'s if necessary, we assume that $a_1<a_2<\dots<a_m$. Then we pick some real numbers $c_1,\dots,c_{m-1}$ such that $a_i<c_i<a_{i+1}$ for every $i$. Since $df(x).X<0$ at every point $x$ which is not chain-recurrent, $c_1,\dots,c_{m-1}$ are regular values of $f$. Let $T_{p+1},\dots,T_{p+\ell}$ be the collection of the connected components of $f^{-1}(c_1)\sqcup\dots\sqcup f^{-1}(c_{m-1})$. Then $T_{p+1},\dots,T_{p+\ell}$ is a finite collection of pairwise disjoint transverse tori embedded in the interior of $M_{\cT}$. Hence $\cT':=\{T_1,\dots,T_{p+\ell}\}$ is 
a finite collection of pairwise disjoint transverse tori. The manifold $M_{\cT'}$, obtained by cutting $M$ along $\cT'$, can also be obtained by cutting $M_{\cT}$ along $\{T_{p+1},\dots,T_{p+\ell}\}$. Using this remark and again the fact that $df(x).X<0$ for every $x$ which is not chain-recurrent, we obtain that the maximal invariant set of $\Lambda_{\cT'}$ is the chain-recurrent set of $X_{M_{\cT}}$. Therefore $\Lambda_{\cT'}=\Lambda_1\sqcup\dots\sqcup\Lambda_m$. Moreover, by construction of $T_{p+1}\sqcup\dots\sqcup T_{p+\ell}$, each connected component of $M_{\cT'}$  contains at most one $\Lambda_i$.

Now, we claim that $\Lambda_1\sqcup\dots\sqcup\Lambda_m=\mathrm{Core}(X)$. On the one hand, $\mathrm{Core}(X)$ is obviously included $\Lambda_1\cup\dots\cup\Lambda_m$ (recall that $\Lambda_1\sqcup\dots\sqcup\Lambda_m$ is the maximal invariant set of the complement of a collection of transverse tori). On the other hand, the periodic orbits of $X$ are dense in $\Lambda_1\sqcup\dots\sqcup\Lambda_m$. So, if  $\mathrm{Core}(X)$ were not equal to $(\Lambda_1\sqcup\dots\sqcup\Lambda_m)$, then there would be a periodic orbit in $(\Lambda_1\sqcup\dots\sqcup\Lambda_k)-\mathrm{Core}(X)$. But this is impossible since $\Lambda_1\sqcup\dots\sqcup\Lambda_k$ is the maximal invariant set of  $M_{\cT'}$, and $T_1\cup\dots\cup T_p$ intersects every periodic orbit in $\mathrm{Core}(X)$. Therefore, $\Lambda_1\sqcup\dots\sqcup\Lambda_m=\mathrm{Core}(X)$.

This completes the proof of item~1: indeed, the equality $\mathrm{Core}(X)=\Lambda_1\sqcup\dots\sqcup\Lambda_m$ shows in particular that $\mathrm{Core}(X)$ is a locally maximal set. Moreover, this equality gives the decomposition of $\mathrm{Core}(X)$ as a finite union of pairwise disjoint basic sets.

In addition, we have found a finite collection of pairwise disjoint transverse tori $\cT':=\{T_1,\dots,T_{p+\ell}\}$, such that $\Lambda_{\cT'}=\mathrm{Core}(X)=\Lambda_1\sqcup\dots\sqcup\Lambda_m$, and such that each connected component of $M_{\cT'}$ contains at most one $\Lambda_i$. Now, consider a maximal sub-collection $\cT''=\{T_{i_1},\dots,T_{i_q}\}$ of $\cT'$, so that the elements of $\cT''$ are pairwise non-parallel. By construction, every torus in $\cT'\setminus\cT''$ is parallel to a torus in $\cT''$.  Denote $V_1,\dots,V_r$ the connected components of $M_{\cT''}$ (by definition, $V_1,\dots,V_r$ is the plug decomposition associated with the collection of tori $\cT''$). According to lemma~\ref{l.non-parallel}, for each $j=1,\dots,r$, the maximal invariant set of $V_j$ coincides with the maximal invariant set of a connected component of $M_{\cT'}$, \emph{i.e.} is either empty or one of the $\Lambda_i$'s. But, since the elements of $\cT''$ are pairwise non-parallel, $V_j$ cannot be diffeomorphic to $\TT^2\times [0,1]$, and therefore (by Lemma~\ref{l.parallel}), the maximal invariant set of $V_j$ cannot be empty. Therefore, $r=m$ and, for $j=1\dots m$, the maximal invariant set of $V_j$ is one of the $\Lambda_i$'s. Of course, re-ordering the $\Lambda_i$'s if necessary, we can assume that $\Lambda_j$ the maximal invariant set of $V_j$. This completes the proof of item~2.
\end{proof}

\section{Fine plug decompositions up to piecewise topological equivalence}

The purpose of this section is to prove Theorem~\ref{t.finitely-many}, \emph{i.e.} to prove that, up to piecewise topological equivalence, there are only finitely many fine plug decomposition of a given Anosov flow. For this purpose, we need to recall a few general things about hyperbolic plugs, and to study what are the particular features of a hyperbolic that arise in a plug decomposition of an Anosov vector field.

\begin{definition}
\label{d.laminations}
Let $(U,X)$ be a hyperbolic plug. Recall that the vector field $X$ is transverse to the boundary $\partial U$. Also recall that we decompose $\partial U$ as $\partial^{in} U\sqcup\partial^{out} U$, where $\partial^{in} U$ (resp. $\partial^{out} U$) is the union of the connected components of $\partial U$ where $X$ pointing inward (resp. outward) $U$. The surface
$\partial^{in} U$ (resp. $\partial^{out} U$) is called the \emph{entrance boundary} (resp. the \emph{exit boundary}) of $U$. We denote by $\cL^s(U,X)$ the set of all points of $\partial^{in} U$ whose forward orbit remain forever in $U$. Similarly, we denote by $\cL^u(U,X)$ the set of all points of $\partial^{out} U$ whose backward orbit remain forever in $U$.The set $\cL^s(U,X)$ (resp. $\cL^u(U,X)$) is called the \emph{entrance lamination} (resp. the \emph{exit lamination}) of $(U,X)$.
\end{definition}

If we denote by $\Lambda$ the maximal invariant set of a hyperbolic plug $(U,X)$, then $\cL^s(U,X)$ is the intersection of the stable lamination $W^s_U(\Lambda)$ and the entrance boundary $\partial^{in} U$, and $\cL^u(U,X)$ is the intersection of the unstable lamination $W^u_U(\Lambda)$ and the exit boundary $\partial^{out} U$. Since $W^s_U(\Lambda)$ and $W^u_U(\Lambda)$ are compact laminations with two-dimensional leaves which are transverse to $\partial U$ (because $X$ is transverse to $\partial U$, it follows that $\cL^s(U,X)$ and $\cL^u(U,X)$ are laminations with one-dimensional leaves.

\begin{definition}
We say that $(U,X)$ has \emph{a trivial connected component} if there is a connected component $U_0$ of $U$ such that the maximal invariant set $\bigcap_{t\in\RR} X^t(U_0)$ is empty.
\end{definition}

Note that, by definition of $\cL^s(U,X)$, the hyperbolic $(U,X)$ has no trivial connected component if and only if  the lamination $\cL^s(U,X)$ intersects every connected component of the entrance boundary $\partial^{in} U$ (or equivalently, if and only if, the lamination $\cL^u(U,X)$ intersects every connected component of the exit boundary $\partial^{out} U$).

The entrance and exit laminations of a hyperbolic plug are very peculiar laminations:

\begin{proposition}[Proposition~2.8 of \cite{BeBoYu}]
\label{p.general-properties-lamination}
Given any hyperbolic plug $(U,X)$,
\begin{enumerate}
\item $\cL^s(U,X)$ has only finitely many closed leaves,
\item every half non-closed leaf of $\cL^s(U,X)$ is asymptotic to a closed leaf (\emph{i.e.} each end of a non-closed leaf of $\cL^s(U,X)$ ``spirals around a closed leaf"),
\item every closed leaf of $\cL^s(U,X)$ can be oriented so that its holonomy is a contraction.
\end{enumerate}
Of course, similar properties hold for $\cL^u(U,X)$.
\end{proposition}

In the present paper, we are interested in a quite particular type of hyperbolic plugs:

\begin{definition}
A hyperbolic plug $(U,X)$ is called \emph{of Anosov type} if there exist a closed three-manifold $M$ and an embedding $\theta:U\hookrightarrow M$, such that $\theta_*X$ is the restriction an Anosov vector field on $M$.
\end{definition}

Of course, all hyperbolic plugs which appear in a plug decomposition of an Anosov vector field are of Anosov type. The entrance and exit laminations of a hyperbolic plug of Anosov type satisfy some specific properties, as shown by Proposition~\ref{p.lamination-Anosov-type} below.

\begin{definition}
\label{d.strip}
Let $(U,X)$ be a hyperbolic plug. A connected component $C$ of $\partial^{in} U\setminus \cL^s(U,X)$ (resp. $\partial^{out} U\setminus \cL^u(U,X)$) is called a \emph{strip} if it is homeomorphic to $\RR\times (0,1)$, and if the accessible boundary of $C$ is made of two non-closed leaves of $\cL^s(U,X)$ (resp. $\cL^u(U,X)$) which are asymptotic to each other at both ends.
\end{definition}

\begin{proposition}
\label{p.lamination-Anosov-type}
Let $(U,X)$ be a hyperbolic plug with no trivial component. Assume that $(U,X)$ is of Anosov type. Then each connected component of $\partial^{in} U\setminus \cL^s(U,X)$ (resp. $\partial^{out} U\setminus \cL^u(U,X)$) is either an annulus bounded by compact leaves of $\cL^s(U,X)$, or a strip in the sense of definition~\ref{d.strip}.
\end{proposition}

\begin{proof}
We begin by constructing a vector field transverse to the lamination $\cL^s(U,X)$.

\medskip

\noindent \textit{Claim. There is a non-singular Morse-Smale vector field $Z$ on $\partial^{in} U$ which is transverse to the lamination $\cL^s(U,X)$.}

\smallskip

Let $\Lambda$ be the maximal invariant set of $(U,X)$. If $(U,X)$ is of Anosov type, then there exists a closed three-manifold $M$, an Anosov vector field $\widehat X$ on $M$, and an embedding $\theta:U\hookrightarrow M$ such that $\theta_*X=\widehat X$. By transversality, the weak stable foliation of the Anosov vector field $\widehat X$ induces a one-dimensional foliation $\cF$ on $\theta(\partial^{in} U)$. Since $\theta(\Lambda)$ is a hyperbolic basic set of $X$, and the lamination $\theta_*\cL^s(U,X)$ is the intersection of the local weak stable lamination $W^s_U(\theta(\Lambda))$ with $\partial^{in} U$. It follows that the foliation $\cF$ contains $\theta_*\cL^s(U,X)$ as a sublamination. Now, we consider a non-singular vector field $Z$ on the surface $\partial^{in} U$ which is transverse to the foliation $(\theta)^{-1}_*(\cF)$ (one may, for example, endow $\partial^{in} U$ with a riemannian metric and consider a unitary vector field orthogonal to $(\theta)^{-1}_*(\cF)$). Since Morse-Smale vector fields are dense in dimension 2, we can perturb $Z$ so that it is a Morse-Smale vector field. We choose the perturbation is small enough, $Z$ is still transverse to $(\theta)^{-1}_*(\cF)$. In particular, $Z$ is transverse to the lamination $\cL^s(U,X)$.  This completes the claim.

\medskip

Let $C$ be a connected component of $\partial^{in} U\setminus \cL^s(U,X)$. Since $(U,X)$ has no trivial connected component, $C$ cannot be a whole connected component of $\partial^{in} U$, \emph{i.e.} the boundary of $C$ is non-empty. As a consequence, the accessible boundary of $C$ is also non-empty. Both the boundary of $C$ and the accessible boundary of $C$ are union of leaves of the lamination $\cL^s(U,X)$.

\medskip

We first consider the case where every leaf of $\cL^s(U,X)$ in the accessible boundary of $C$ is a closed leaf. Since $\cL^s(U,X)$ has only a finite number of closed leaves (Proposition~\ref{p.general-properties-lamination}), the boundary of $C$ coincides with the accessible boundary, \emph{i.e.} the boundary of $C$ is made of closed leaves of $\cL^s(U,X)$. It follows that $C$ is homeomorphic to the interior of a compact surface $\overline{C}$. According to the claim above, the compact surface $\overline{C}$ carries a non-singular vector field, which is transverse to its boundary. Hence the Euler characteristic of $C$ is equal to $0$, \emph{i.e.} $C$ is an annulus.

\medskip

Now we consider the case where there is a non-closed leaf $\gamma$ in the accessible boundary of $C$. Replacing the vector field $Z$ by $-Z$ if necessary, we may ---\;and we do\;--- assume that $Z$ is pointing inwards $C$ along $\gamma$.

\medskip

\noindent \textit{Claim. For every $x\in\gamma$, the forward orbit of $x$ under the flow of $Z$ eventually exits $C$.}

\smallskip

Let $E$ be the set made of the points $x\in\gamma$ such that the forward orbit of $x$ under the flow of $Z$ eventually exits $C$ (\emph{i.e.} such that there exists $t>0$ such that $Z^t(x)\notin C$). The forward orbit  a point $x\in\gamma$ eventually exists $C$ if and only if it intersects the lamination $\cL^s(U,X)$. Since the vector field $Z$ is transverse to the lamination $\cL^s(U,X)$, it follows that $E$ is an open subset of $\gamma$. On the other hand, if the forward of a point $x\in\gamma$ remains in $C$ forever, then the forward orbit of $x$ is included in the basin of an attracting periodic orbit of $Z$ contained in $C$ (recall that $Z$ is a non-singular  Morse-Smale vector field, hence every forward orbit is attracted by a periodic orbit). Hence, $\gamma\setminus E$ coincides with the intersection of $\gamma$ with the union of the basins of the attracting periodic orbits of $Z$ contained in $C$. In particular, $\gamma\setminus E$ is an open subset of $\gamma$. So we have proved that $E$ is an open and closed subset of $\gamma$. To complete the proof of the claim, it remains to show that $E$ is non-empty. The ends of $\gamma$ spiral around some closed leaves $\alpha_1,\alpha_2$ of $\cL^s(U,X)$. The vector field $Z$ is transverse to $\alpha_1,\alpha_2$. Hence, in some neighborhood of $\alpha_i$, every orbit of $Z$ cuts $\gamma$ infinitely many times. It follows that, for $x\in\gamma$ close enough to $\alpha_i$, the forward orbit of $x$ will intersect $\gamma$, and therefore exit $C$. In other words, the set $E$ contains some neighborhood of the ends of $\gamma$. This completes the proof of the claim.

\medskip

Now we prove that $C$ is a strip. For every point $x\in\gamma$, let $\tau(x)=\inf\{t>0, Z^t(x)\notin C\}$. The above claim ensures that $\tau(x)$ is finite for every $x\in\gamma$. Since the boundary of $C$ is made of leaves of the lamination $\cL^s(U,X)$ and since $Z$ is transverse to $\cL^s(U,X)$, the time $\tau(x)$ must depend continuously on $x$. Moreover, since the ends of $\gamma$ spirals around closed leaves and since $Z$ is transverse to these closed leaves, $\tau(x)$ must tend to $0$ when $x$ approaches the ends of $\gamma$. We consider the map $\phi:\gamma\times (0,1)\to C$ defined by $\phi(x,t):=Z^{t/\tau(x)}(x)$. This map is continuous (because the flow of $Z$ and $x\mapsto \tau(x)$ are continuous), one-to-one (otherwise this would imply that the forward orbit of some point $x\in\gamma$ intersects $\gamma$ before exiting $C$, which is of course impossible), and proper (since $\tau(x)$ tends to $0$ when $x$ approaches the ends of $\gamma$). It follows that $\phi$ is a homeomorphism. In particular, $C$ is homeomorphic to $\RR\times (0,1)$. Now, by transversality, the forward orbits of all the points of $\gamma$ exit $C$ through the same leaf $\gamma'$. By construction, this leaf is contained in the accessible boundary of $C$. It is non-compact, since $x\mapsto Z^{\tau(x)}(x)$ defines a homeomorphism between $\gamma$ and $\gamma'$. Moreover $\gamma$ and $\gamma'$ are asymptotic to each other at both ends, since $\tau(x)$ tends to $0$ when $x$ approaches the ends of $\gamma$. One easily deduces that the accessible boundary of $C$ is reduced to $\gamma\cup\gamma'$. Hence, $C$ is a strip in the sense of definition~\ref{d.strip}. This completes the proof of the proposition.
\end{proof}

Now we will quickly recall the main results of~\cite{BeguinBonatti2002}. More precisely, we will define some particular hyperbolic plugs (called \emph{models}) and explain how an arbitrary hyperbolic plug $(U,X)$ can be obtained from a model by some simple surgeries called \emph{handle attachments}.

\begin{definition}
\label{d.germ}
Let $X, X'$ be vector fields on some three-manifolds (possibly with boundary) $M, M'$, and $\Lambda, \Lambda'$ be some compact invariant sets for $X, X'$ respectively. We say that the germ of $X$ along $\Lambda$ coincides with the germ of $X'$ along $\Lambda'$ if there exist neighborhoods $O, O'$ of $\Lambda, \Lambda'$ respectively, such that $X_{|O}$ is topologically equivalent to $X'_{|O'}$.
\end{definition}

\begin{definition}
\label{d.model}
Let $X$ be a vector field on a closed three-manifold, and $\Lambda$ be a hyperbolic basic set of $X$. A \emph{model} of the germ of $X$ along $\Lambda$ is a hyperbolic plug $(\widetilde U,\widetilde X)$ with maximal invariant set $\widetilde\Lambda$ such that~:
\begin{itemize}
\item  $(\widetilde U,\widetilde X)$ has no trivial connected component;
\item the germ of $\widetilde X$ along $\widetilde\Lambda$ coincides with the germ of $X$ along $\Lambda$;
\item every simple closed curve in $\partial^{in}\widetilde U\setminus \cL^s(\widetilde U,\widetilde X)$ bounds a disc in  $\partial^{in}\widetilde U\setminus \cL^s(\widetilde U,\widetilde X)$.
\end{itemize}
\end{definition}

One of the main results of~\cite{BeguinBonatti2002} is the following:

\begin{theorem}[Theorem 0.3 of \cite{BeguinBonatti2002}]
\label{t.model}
For every vector field $X$ on a closed three-manifold, and every hyperbolic basic set $\Lambda$ of $X$, the model of the germ of a vector field $X$ along a hyperbolic set $\Lambda$ exists and is unique up to topological equivalence.
\end{theorem}

Let $(U,X)$ be a orientable hyperbolic plug. Given any set $A\subset U$, we denote by $\cO_X(A)$ the orbit of $A$ (in $U$) under the flow of $X$. Let $\{D_j^k\}_{j=1\dots\ell}^{k=1,2}$ be a collection of $2\ell$ pairwise disjoint closed discs in $\partial^{in} U\setminus \cL^s(U,X)$. Since each $D_j^k$ is disjoint from the lamination $\cL^s(U,X)$, for every point $x$ in $\cup_{j,k} D_j^k$, the forward orbit of $x$ eventually exits $U$ (by cutting $\partial^{out} U$). Moreover, since the surfaces $\partial^{in} U$ and $\partial^{out} U$ are transverse to the vector field $X$, the exit time of the orbit of $x$ depends continuously on $x$ (for $x$ in $\cup_{j,k} D_j^k$). Therefore, up to topological equivalence, we can ---\;and we will\;--- assume that the exit time of the forward orbit of every point $x$ in $\cup_{j,k} D_j^k$ is equal to $1$. After this topological equivalence, the restriction of $X$ to the orbit $\cO_X(D_j^k)$ is conjugated to the trivial vector field $\frac{\partial}{\partial t}$ on the product space $D_j^k\times [0,1]$.

Let $U_0:=U\setminus \bigcup_{j,k} \cO_X(\mathrm{int}(D_j^k))$. The above discussion shows that ``$U_0$ is obtained by digging $2\ell$ tunnels going from $\partial^{in} U$ to $\partial^{out} U$".  Now, we will ``glue pairwise the boundaries of these $2\ell$ tunnels".  See figure~\ref{f.handle-attachment}.

Consider a diffeomorphism $\varphi:\bigcup_j \partial D_j^1\to\bigcup_j \partial D_j^2$, so that $\varphi$ maps $\partial D_j^1$ to $\partial D_j^2$, and  reverses orientation when the curves $\partial D_j^1,\partial D_j^2$ are equipped with their orientations as boundaries.
Consider the diffeomorphism $\Phi:\bigcup_{j} \cO_X(\partial D_j^1) \to \bigcup_{j} \cO_X(\partial D_j^2)$ defined by $\Phi(X^t(x))=X^t(\varphi_j(x))$ for $x\in \partial D_j^1$. Note that this diffeomorphism is indeed well-defined and onto because we have assumed that the exit time of the orbit of every point $x$ in $\bigcup_{j} \cO_X(\partial D_j^1)$ and $\bigcup_{j} \cO_X(\partial D_j^2)$ is equal to $1$. Also note that $\Phi$ preserves the vector field $X$. Let $V:=U_0/\Phi$, and denote by $Y$ the vector field induced by $X$ on $V$ (see figure~\ref{f.handle-attachment}). One easily checks that $V$ is a manifold with boundary and that the vector field $Y$ is transverse to the boundary of $V$. Moreover, since the orbits of the discs $D_1^1,D_1^2,\dots,D_n^1,D_n^2$ are disjoint from the maximal invariant set $\bigcap_{t} X^t(U)$, the germ of vector field $Y$ along the hyperbolic set $\bigcap_{t} Y^t(V)$ coincides with the germ of the vector field $X$ along the hyperbolic set $\bigcap_{t} X^t(U)$. In particular, $(V,Y)$ is a hyperbolic plug.

\begin{definition}[Handle attachment]
\label{d.handle-attachment}
We say that $(V,Y)$ is obtained from $(U,X)$ by handle attachments on the pair of discs $(D_1^1,D_1^2), (D_2^1,D_2^2),\dots, (D_\ell^1,D_\ell^2)$.
\end{definition}

\begin{remark}
\label{r.boundary-handle-attachment}
Clearly, the entrance boundary $\partial^{in} V$ is given by $\partial^{in} V=(\partial^{in} U\setminus \bigcup_{j,k} \mathrm{int}(D_j^k))/\varphi$. One easily deduces that $\partial^{in} V$ is obtained by attaching $n$ handles on the surface $\partial^{in} U$ in the classical topological sense (see figure~\ref{f.handle-attachment}).
\end{remark}

\begin{figure}[h]
\centerline{\includegraphics[totalheight=5.4cm]{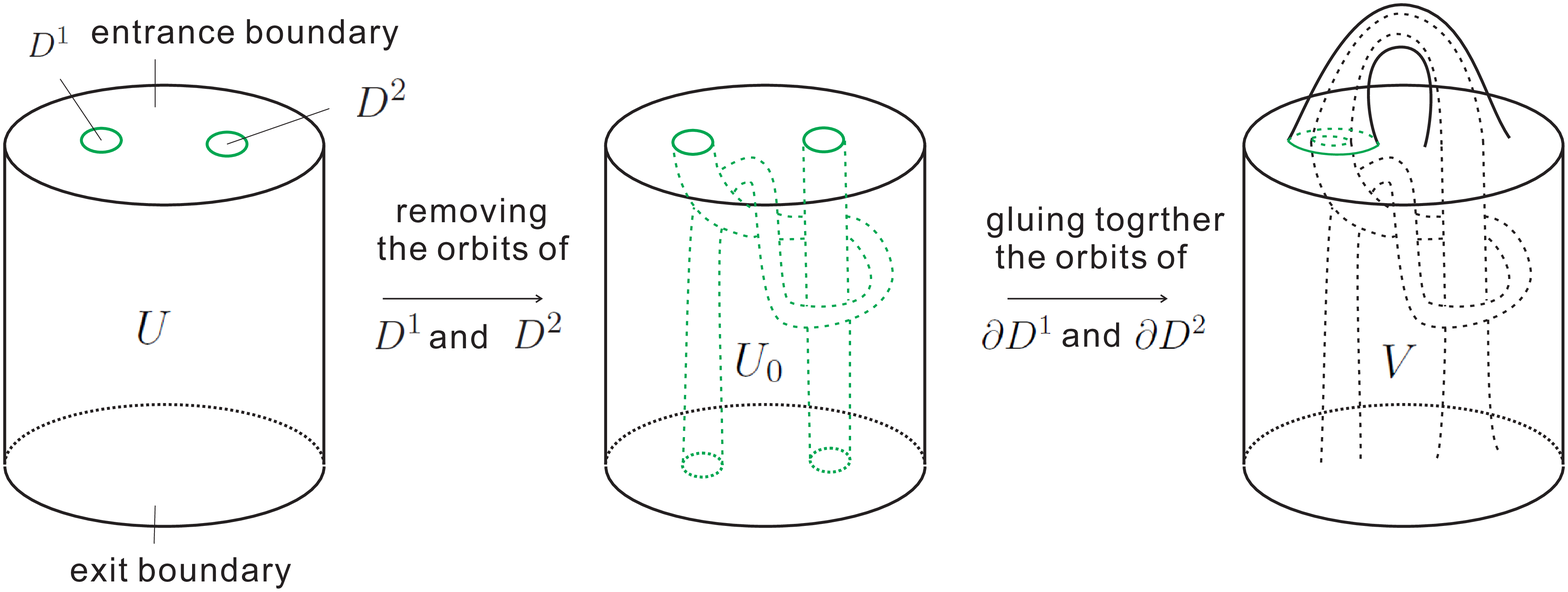}}
\caption{\label{f.handle-attachment}Handle attachment on a pair of discs $(D^1,D^2)$}
\end{figure}

\begin{proposition}[Lemma 3.2 of \cite{BeguinBonatti2002}]
Up to topological equivalence, $(V,Y)$ only depends on the connected components of $\partial^{in} U\setminus \cL^s(U,X)$ containing the discs $D_1^1,D_1^2,\dots,D_\ell^1,D_\ell^2$.
\end{proposition}

In the two following statements (Proposition~\ref{p.model-to-any} and~\ref{p.neighbourhood-Anosov-type}), we consider a hyperbolic plug $(U,X)$ with no trivial connected component. We denote by $(\widetilde U,\widetilde X)$ be the model of the germ of $X$ along the hyperbolic set $\Lambda:=\bigcap_{t} X^t(U)$ (see Definition~\ref{d.model} and Theorem~\ref{t.model}). We assume that $(U,X)$ has no trivial component (\emph{i.e.} every connected component of $U$ meets $\Lambda$). The first proposition below states that $(U,X)$ can be obtained from the model $(\widetilde U,\widetilde X)$ thanks to a finite number of handle attachments. The second proposition states some very strong restriction on the position of these handle attachments, in the particular case where $(U,X)$ is of Anosov type.

\begin{proposition}[Proposition 0.5 of \cite{BeguinBonatti2002}]
\label{p.model-to-any}
There exist a finite family of pairwise disjoint topological closed discs $D_1^1,D_1^2,D_2^1,D_2^2,\dots,D_\ell^1,D_\ell^2$ in $\partial^{in}\widetilde U\setminus \cL^s(\widetilde U,\widetilde X)$ such that, up to topological equivalence, $(U,X)$ can be obtained from $(\widetilde U,\widetilde X)$ by handle attachments on the pair of discs $(D_1^1,D_1^2),\dots,(D_\ell^1,D_\ell^2)$.
\end{proposition}

\begin{proposition}
\label{p.neighbourhood-Anosov-type}
Let $D_1^1,D_1^2,\dots,D_\ell^1,D_\ell^2$ be some discs provided by Proposition~\ref{p.model-to-any}. Denote by $C_j^k$ the connected components of $\partial^{in}\widetilde U\setminus \cL^s(\widetilde U,\widetilde X)$ containing the disc $D_j^k$. If $(U,X)$ is of Anosov type, then:
\begin{enumerate}
\item for every $j,k$, the connected component $C_j^k$ is a disc bounded by a closed leaf of $\cL^s(\widetilde U,\widetilde X)$~;
\item the connected components $C_1^1,C_1^2,\dots,C_\ell^1,C_\ell^2$ are pairwise different.
\end{enumerate}
\end{proposition}

\begin{proof}
Let $\varphi:\bigcup_j \partial D_j^1\to\bigcup_j \partial D_j^2$ be a diffeomorphism mapping $\partial D_j^1$ to $\partial D_j^2$, and  reversing orientation when the curves $\partial D_j^1,\partial D_j^2$ are equipped with their orientations as boundaries, as considered in Definition~\ref{d.handle-attachment}. Let $\pi$ be the projection of the compact surface with boundary $\partial^{in}\widetilde U\setminus \bigcup_{j,k} \mathrm{int}(D_j^k)$ on the closed surface $(\partial^{in}\widetilde U\setminus \bigcup_{j,k} \mathrm{int}(D_j^k))/\varphi$.

Up to topological equivalence, the neighborhood $(U,X)$ is obtained from $(\widetilde U,\widetilde X)$ by handle attachments on the pairs of discs $(D_1^1,D_1^2),\dots,(D_\ell^1,D_\ell^2)$. Considering the restriction of the topological equivalence to the entrance boundary of $U$, we obtain a homeomorphism $h$ between the surfaces $(\partial^{in}\widetilde U\setminus \bigcup_{j,k} \mathrm{int}(D_j^k))/\varphi$ and $\partial^{in} U$ (see remark~\ref{r.boundary-handle-attachment}). Since $h$ is the restriction of a topological equivalence, $h\circ \pi$ maps the lamination $\cL^s(\widetilde U,\widetilde X)$ on the lamination $\cL^s(U,X)$.

For every $i$, let $\Gamma_j:=h(\pi(\partial D_j^1))=h(\pi(\partial D_j^2))$. Note that $\Gamma_1,\dots,\Gamma_\ell$ are pairwise disjoint simple closed curves in $\partial^{in} U\setminus \cL^s(U,X)$. In particular, $\partial^{in} U\setminus \bigcup_j\Gamma_j$ is the interior of a compact surface with boundary. Also note that $h\circ\pi$ induces a homeomorphism between the open surfaces $\partial^{in}\widetilde U\setminus \bigcup_{j,k} D_j^k$ and $\partial^{in} U\setminus\bigcup_j \Gamma_j$.

\medskip

\textit{Claim. Each connected component of $\partial^{in} U\setminus \bigcup_j \Gamma_j$ intersects the lamination $\cL^s(U,X)$.}

\smallskip

Indeed, suppose that there exists a connected component $B$ of $\partial^{in} U\setminus \bigcup_j\Gamma_j$ which does not intersect $\cL^s(U,X)$. Then $(h\circ\pi)^{-1}(B)$ is a connected component of $\partial^{in}\widetilde U\setminus \bigcup_{j,k} D_j^k$ which does not intersect the lamination $\cL^s(\widetilde U,\widetilde X)$. Since the discs $D_j^k$'s are also disjoint from the lamination $\cL^s(\widetilde U,\widetilde X)$, it follows that $(h\circ\pi)^{-1}(B)$ is contained in a connected component of $\partial^{in} U$ which is disjoint from the lamination $\cL^s(\widetilde U,\widetilde X)$. This is impossible since $(\widetilde U,\widetilde X)$ has no trivial connected component (see Definition~\ref{d.model}).

\medskip

Now, denote by $A_i$ the connected component of $\partial^{in} U\setminus\cL^s(U,X)$ containing the curve $\Gamma_i$. The claim proved above implies that none of the curves $\Gamma_1,\dots,\Gamma_\ell$ is homotopic to $0$ in $\partial^{in} U\setminus\cL^s(U,X)$ and that the curves $\Gamma_1,\dots,\Gamma_\ell$ are pairwise non-homotopic in $\partial^{in} U\setminus\cL^s(U,X)$. Together with Proposition~\ref{p.lamination-Anosov-type} (recall that $(U,X)$ is assumed to be of Anosov type), this implies that, for every $i=1\dots\ell$, the connected component $A_i$ is an annulus bounded by to two closed leaves of the lamination $\cL^s(U,X)$, and that the annuli $A_1,\dots,A_\ell$ are pairwise different. As a consequence,
\begin{enumerate}
\item $A_j\setminus \Gamma_j$ is disjoint from $\bigcup_j \Gamma_j$~;
\item $A_j\setminus \Gamma_j$ has two connected components $A_j^1$ and $A_j^2$~;
\item each of these two connected components $A_j^1,A_j^2$ is an annulus bounded by $\Gamma_j$ and a closed leaf of $\cL^s(U,X)$.
\end{enumerate}
Recall that $C_j^k$ is the connected component of $\partial^{in} \widetilde U\setminus\cL^s(\widetilde U,\widetilde X)$ containing the disc $D_j^k$. Item 1,2,3 above imply that $(h\circ\pi)^{-1}$ is defined on the whole connected component $A_j^k$, and that $C_j^k$ is obtained by gluing the disc $D_j^k$ on one of $(h\circ\pi)^{-1}(Aj^k)$ (up to exchanging the names of $A_j^1$ and $A_j^2$). Using item~3 again, it follows that $C_j^k$ is a disc bounded by a compact leaf of $\cL^s(\widetilde U,\widetilde X)$. It also follows that, for $(j',k')\neq(j,k)$, the disc $D_{j'}^{k'}$ is disjoint from $C_j^k$~; equivalently, the connected components $C_1^1,C_1^2,\dots,C_\ell^1,C_\ell^2$ are pairwise different. This completes the proof of Proposition~\ref{p.neighbourhood-Anosov-type}.
\end{proof}

\begin{proof}[Proof of Theorem~\ref{t.finitely-many}]
Let $X$ be an Anosov vector field on a closed orientable three-manifold $M$. Let $\Lambda_1,\dots,\Lambda_m$ be the residual basic sets  provided by Theorem~\ref{t.main}. For each $i=1,\dots,m$, let $(\widetilde V_i,\widetilde X_i)$ be the model of the germ of $X$ along $\Lambda_i$.

By definition, if $V_1,\dots,V_m$ is a fine plug decomposition of $(M,X)$, then (up to permutation of the indices) $(V_i,X_{|V_i})$ is a connected hyperbolic plug with maximal invariant set $\Lambda_i$. Note that $(V_i,X_{|V_i})$ has no trivial connected component (since it is connected and has a non-empty maximal invariant set). Therefore, according to Proposition~\ref{p.model-to-any}, we can find discs $D_{i,1}^1,D_{i,1}^2,\dots,D_{i,\ell_i}^1,D_{i,\ell_i}^2$ such that, up to topological equivalence, $(V_i,X_{|V_i})$ can be obtained from $(\widetilde V_i,\widetilde X_i)$ by handle attachments on the pair of discs $(D_{i,1}^1,D_{i,1}^2),\dots,(D_{i,\ell_i}^1,D_{i,\ell_i}^2)$. Denote by $C_{i,j}^k$ the connected component of $\partial^{in} \widetilde V_i\setminus \cL^s(\widetilde V_i,\widetilde X_i)$ containing the disc $D_{i,j}^k$. By construction, $(V_i,X_{|V_i})$ is  of Anosov type (since it belongs to a plug decomposition of an Anosov vector field). So, by Proposition~\ref{p.neighbourhood-Anosov-type}, the connected components $C_{i,1}^1,C_{i,1}^2,\dots,C_{i,\ell_i}^1,C_{i,\ell_i}^2$ are pairwise different, and each of them is a topological disc bounded by closed leaf of the lamination $\cL^s(\widetilde V_i,\widetilde X_i)$. According to proposition~\ref{p.general-properties-lamination}, the lamination $\cL^s(\widetilde V_i,\widetilde X_i)$ has only finitely many closed leaves. It follows that the integer $\ell_i$ is a priori bounded, and that there are only finitely many possible choices for the connected components $C_{i,1}^1,C_{i,1}^2,\dots,C_{i,\ell_i}^1,C_{i,\ell_i}^2$. Moreover, Theorem~\ref{t.model} and~\ref{p.model-to-any} state that, up to topological equivalence, the hyperbolic plug $(V_i,X_{|V_i})$ only depends on the components $C_{i,1}^1,C_{i,1}^2,\dots,C_{i,\ell_i}^1,C_{i,\ell_i}^2$. Therefore, up to topological equivalence, there are only finitely many possibilities for the hyperbolic plug $(V_i,X_{|V_i})$. This completes the proof of Theorem~\ref{t.finitely-many}
\end{proof}

Now, we describe a particular situation, where a fine hyperbolic plug decomposition is unique. The following definition was introduced in~\cite{BeBoYu}.

\begin{definition}
\label{d.filling}
Let $(U,X)$ be an orientable hyperbolic plug. We say that $(U,X)$ has \emph{filling laminations} if every connected components of $\partial^{in} U\setminus\cL^s(U,X)$ is a strip in the sense of Definition~\ref{d.strip}.
\end{definition}

Recall that, in the particular case where $(U,X)$ is of Anosov type, Definition~\ref{d.filling} is equivalent to requiring that no connected components of $\partial^{in} U\setminus\cL^s(U,X)$ is an annulus bounded by closed leaves of  $\cL^s(U,X)$ (Proposition~\ref{p.neighbourhood-Anosov-type}). Also observe that, if $(U,X)$ has filling laminations, then every connected components of $\partial^{in} U\setminus\cL^s(U,X)$ is contractible, and therefore, $(U,X)$ is a model of the germ of $X$ along the hyperbolic set $\bigcap_t X^t(U)$.

\begin{proposition}
\label{p.filling-lamination-case}
Let $X$ be an Anosov vector field on a closed orientable three-manifold $M$. If $(M,X)$ admits  a fine plug decomposition $V_1^0,\dots,V_m^0$ so that $(V_i^0,X_{|V_i^0})$ has filling laminations for every $i$, then this fine plug decomposition is unique up to piecewise topological equivalence.
\end{proposition}

\begin{proof}
Let $V_1,\dots,V_m$ be an arbitrary fine plug decomposition of $(M,X)$. Note that since $(V_i^0,X_{|V_i^0})$ has filling laminations, every connected components of $\partial^{in} V_i^0\setminus\cL^s(V_i^0,X)$ is contractible, and therefore, $(V_i^0,X_{|V_i^0})$ is a model of the germ of $X$ along the hyperbolic set $\Lambda_i:=\bigcap_t X^t(V_i^0)$. Hence, for every $i$, up to topological equivalence, the hyperbolic plug $(V_i,X_{|V_i})$ is obtained from by some handles attachment $(V_i^0,X_{|V_i^0})$ (see the proof of Theorem~\ref{t.finitely-many}). But the arguments of the proof Theorem~\ref{t.finitely-many} show that each handle attachment correspond to a pair of connected components $C_{i,j}^1,C_{i,j}^2$ of $\partial^{in} V_i^0\setminus \cL^s(V_i^0,X_{|V_i^0})$ which both are topological discs bounded by closed leaves of $\cL^s(V_i^0,X_{|V_i^0})$. Since $(V_i^0,X_{|V_i^0})$ has filling laminations, such connected components do not exists. Hence, there is actually no handle attachment, \emph{i.e.} $(V_i,X_{|V_i})$ is topological equivalent to $(V_i^0,X_{|V_i^0})$. By definition, this means that the plug decompositions $V_1,\dots,V_m$ and $V_1^0,\dots,V_m^0$ are piecewise topologically equivalent.
\end{proof}

\section{Fine plug decompositions up to flow isotopy equivalence}

The purpose of this section is to build an example to prove Proposition \ref{p.infinitely-many}, \emph{i.e.} to build an Anosov vector field $Z$ on a closed manifold $M$, so that $Z$ admits with infinitely many fine plug decompositions which are pairwise not flow isotopy equivalent. Actually, the manifold $M$ and the vector field $Z$ were already constructed in the final section of our paper \cite{BeBoYu}. In that paper, we proved the existence of infinitely many pairwise non-isotopic tori transverse to $Z$. Here, we only briefly recall the construction of $(M,Z)$, and explain what are the fine plug decompositions.

\bigskip

First we consider the vector field $X_0$ on the torus $\TT^2$ defined  as follows
$$X_0(x,y)=\sin(2\pi y)\frac{\partial}{\partial x}+\sin(2\pi y)\frac{\partial}{\partial y}.$$
The non-wandering set of $X_0$ consists in four hyperbolic singularities: a source~$\alpha:=(0,0)$, two saddles $\sigma_1:=(\frac{1}{2},0)$ and $\sigma_2:=(0,\frac{1}{2})$, and a sink $\omega:=(\frac{1}{2},\frac{1}{2})$. Moreover,  the invariant manifolds of $\sigma_1$ are disjoint from the invariant manifold of $\sigma_2$. See figure~\ref{f.gradient-like}.

\begin{figure}[ht]
\begin{center}
  \includegraphics[totalheight=5cm]{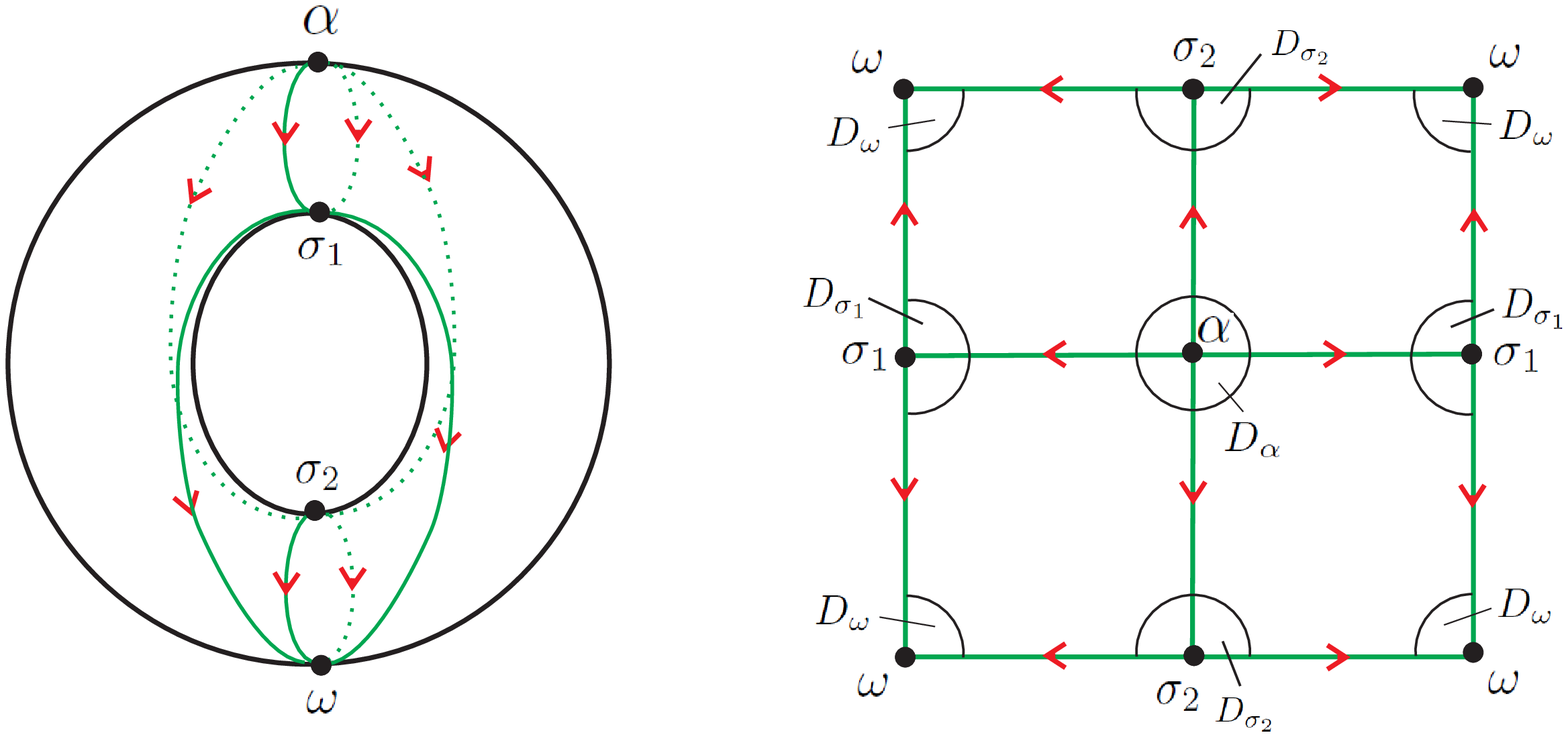}
  \caption{\label{f.gradient-like}The gradient-like vector field $X_0$}
  \end{center}
\end{figure}

Now, let us choose four small and standard open discs neighborhoods $D_\alpha,D_{\sigma_1},D_{\sigma_2},D_{\omega}$ of $\alpha,\sigma_1,\sigma_2,\omega$ respectively. We consider a smooth function $\varphi:\TT^2\to\RR$ such that $\varphi>0$ on $D_{\sigma_1}$, $\varphi<0$ on $D_{\sigma_2}$, and $\varphi=0$ on $\TT^2\setminus (D_{\sigma_1}\cup D_{\sigma_2})$ (in particular, $\varphi=0$ on $D_\alpha\cup D_\omega$). Then we consider the vector field $X$ on $\TT^2\times \SS^1$ defined by
$$X(x,y,t)= X_0(x,y)+ \varphi(x,y)\frac\partial{\partial t}.$$
We consider the compact three-manifold with boundary $U:=(\TT^2\setminus (D_\alpha\cup D_\omega))\times \SS^1$.
Then $(U,X)$ is  a hyperbolic plug. The maximal invariant set of $(U,X)$ consists in two saddle hyperbolic periodic orbits. The entrance boundary $\partial^{in} U$ (resp. the exit boundary $\partial^{out} U$) is the torus $\partial D_\alpha\times \SS^1$ (resp. $\partial D_\omega\times \SS^1$).

Now we consider two copies $(U_1,X_1)$ and $(U_2,X_2)$ of the plug $(U,X)$. We choose a diffeomorphism $\psi:\partial^{out} U_1\to \partial^{in} U_2$ (see Definition~\ref{d.laminations}) so that $(V,Y):=(U_1\sqcup U_2, X_1\sqcup X_2)/\psi$ is a hyperbolic plug with boundary composed of two tori. In \cite{BeBoYu}, we proved that there exists a diffeomorphism $\chi:\partial^{out} V\to \partial^{in} V$  such that $(M,Z):= (V,Y)/\chi$ is a transitive Anosov flow (see lemma 11.4 of \cite{BeBoYu}, and the paragraph following the proof of this lemma).

The manifold $M$ constructed above is a graph manifold (it was obtained by gluing together two copies of $\Sigma\times \SS^1$ where $\Sigma$ is the torus minus two discs), with two JSJ tori $T=\pi(\partial^{in} U_1)=\pi(\partial^{out} U_2)$ and $T'=\pi(\partial^{out} U_1)=\pi(\partial^{in}U_2)$ (where $\pi$ denotes the projection of $U_1\sqcup U_2$ on $M$).

\begin{proof}[Proof of Proposition~\ref{p.infinitely-many}]
We shall describe infinitely many fine plug decompositions of $(M,Z)$ which are pairwise not flow isotopy equivalent.

Let  $c_x$ and $c_y$ be the closed curves on $\TT^2$ defined respectively by the equations $x=\frac{1}{4}$ and $y=\frac{1}{4}$. We endow $c_x$ and $c_y$ with the orientations defined by the vector fields $\frac{\partial}{\partial y}$ and $\frac{\partial}{\partial x}$ respectively. One can easily check that the vector field $X_0$ is transverse to $c_x$ and $c_y$. One can choose a simple closed curve $b$ in $\TT^2$ which satisfies the following conditions (see the left side of Figure \ref{f.transverse-curves}).
\begin{enumerate}
  \item $b$ is transverse to $X_0$ and $b$ is disjoint with $c_y$, $D_\alpha$, $D_{\sigma_1}$, $D_{\sigma_2}$  and $D_{\omega}$.
  \item The geometrical intersectional number of $b$ and $c_x$ is $1$.
  \item $c_y$ and $b$ cut $\TT^2\setminus (D_\alpha\cup D_\omega))$ in two pants and there is exactly one saddle in each pant.
\end{enumerate}

Let $D$ be a Dehn twist about the curve $-c_x$. For $n\in \NN^+$, let $b_n$ and $c_{y,n}$ be the images of the curves  $b$ and $c_y$ under the $n^{th}$ power of $D$ (the right hand side of Figure \ref{f.transverse-curves} displays $b_n$ and $c_{y,n}$ for $n=1$). If $D$ is supported in a sufficiently small neighborhood of $c_x$, the closed curves $b_n$ and $c_{y,n}$ are transverse to $X_0$. Note that $c_{y,n}$ and $b_n$ also cut $\TT^2\setminus (D_\alpha\cup D_\omega))$ in two pants and there is exactly one saddle in each pant.

\begin{figure}[ht]
\begin{center}
  \includegraphics[totalheight=5.5cm]{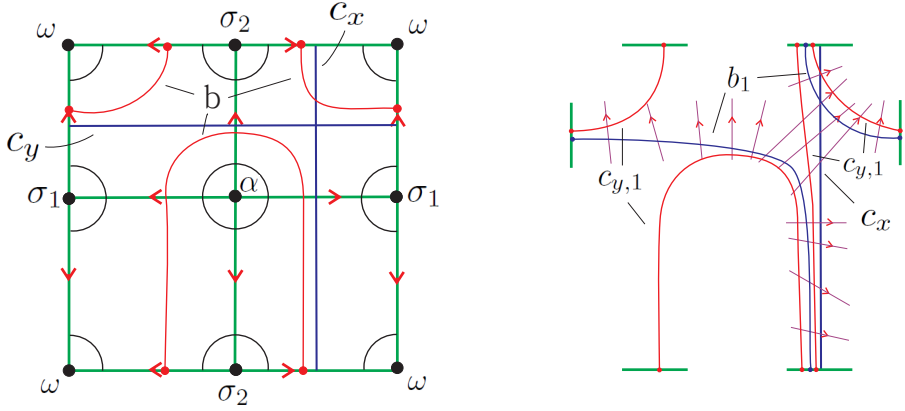}
  \caption{\label{f.transverse-curves} The transverse curves: left, $b$, $c_x$ and $c_y$; right, $b_1$ and $c_{y,1}$}
  \end{center}
\end{figure}

If $c$ is  a simple closed curve in $\TT^2\setminus (D_\alpha\cup D_\omega))$ transverse to $X_0$, then the torus $T_{c}:=c\times\SS^1$ is embedded in $U=(\TT^2\setminus (D_\alpha\cup D_\beta))\times\SS^1$ and transverse to the vector field $X$ (because $c$ is transverse to $X_0$ and  $X(x,y,t)=X_0(x,y)$ for $(x,y)\in \TT^2\setminus (D_{\sigma_1},D_{\sigma_2})$). In particular, for every $n\in\NN^+$, the tori $T_{b_n}:=b_n\times\SS^1$ and $T_{c_{y,n}}$ are transverse to $X$ in $U$.

 Now recall that $(M,Z)$ has been obtained by gluing together two copies $(U_1,X_1)$ and $(U_2,X_1)$ of $(U,X)$ along their boundary. For every curve $c$ as above, we denote by $T_c^1$ (resp. $T_c^2$) the torus $T_c$ seen in $U_1$ (resp. in $U_2$). Recall that the projections in $M$ of the boundary components of $U_1,U_2$ yields two JSJ tori $T,T'$ in $M$.  For every $n\in \NN^+$, we consider the collection of six tori set $\cT_n =\{T,T',T_b^1,T_{c_y}^1,T_{b_n}^2, T_{c_{y,n}}^2\}$.

 By the construction, $T,T',T_b^1,T_{c_y}^1,T_{b_n}^2, T_{c_{y,n}}^2$ are pairwise disjoint, pairwise non-isotopic and transverse to $Z$. Therefore, the collection $\cT_n$ defines a plug decomposition of $(M,Z)$. To prove Proposition~\ref{p.infinitely-many}, we should prove the following two claims:
 \begin{enumerate}
   \item the plug decomposition associated to $\cT_n$ is a fine plug decomposition;
   \item if $m\neq n$, the plug decompositions associated to $\cT_n$ and $\cT_m$ are not flow isotopy equivalent.
 \end{enumerate}

 For every $n$, the plug decomposition associated to $\cT_N$ is made of four connected hyperbolic plugs with maximal invariant sets $s_1^1$, $s_2^1$, $s_1^2$ and $s_2^2$  respectively. Here $s_1^i$ and $s_2^i$ are the isolated saddle periodic orbits associated to the singularities $\sigma_1$ and $\sigma_2$ in $(U_i,X_i)$.  Therefore, to prove the claim~1, we only need to prove that $\mathrm{Core}(X)= \{s_1^1, s_2^1,s_2^1,s_2^2\}$. For this purpose, recall  that:
 \begin{itemize}
 \item $M$ is a graph manifold, with two Seifert pieces, each of which is a copy $\Sigma\times \SS^1$, where $\Sigma$ is the torus minus two discs,
 \item each of the four orbits $s_1^1,s_1^2,s_2^1,s_2^2$ is isotopic to a fiber of one of the Seifert pieces of $M$.
 \end{itemize}
It follows that, for every torus $T$ embedded in $M$, the algebraic intersection number of $T$ with any of the four orbits $s_1^1,s_1^2,s_2^1,s_2^2$ is equal to $0$. If $T$ is transverse to $Z$, the algebraic intersection number of $T$ with any orbit of $Z$ coincides with the geometric intersection number. This proves that every torus $T$ transverse to $Z$ is disjoint from $s_1^1, s_2^1,s_2^1,s_2^2$. Hence, $\mathrm{Core}(X)= \{s_1^1, s_2^1,s_2^1,s_2^2\}$.

 Now we turn to prove claim~2. It is enough to prove that there is no self-homeomorphism of $M$, isotopic to the identity, and mapping the collection of tori $\cT_m$ on the collection of tori $\cT_n$.  Using the topological structure of $M$, this reduces to proving that the pairs of simple closed curves $\{b_m, c_{y,m}\}$ and  $\{b_n, c_{y,n}\}$ are not isotopic in $\TT^2\setminus (D_\alpha\cup D_\omega))$. This last fact clearly follows from the definition of the curves $b_m,b_n,c_{y,m},c_{y,n}$.
 \end{proof}

\vskip 1cm
\noindent Fran\c cois B\'eguin

\noindent {\small Laboratoire Analyse, G\'eom\'etrie, Applications - UMR 7539 du CNRS}

\noindent{\small Universit\'e Paris 13, 93430 Villetaneuse, FRANCE}

\noindent{\footnotesize{E-mail: beguin@math.univ-paris13.fr}}
\vskip 2mm

\noindent Christian Bonatti,

\noindent {\small Institut de Math\'ematiques de Bourgogne - UMR 5584 du CNRS}

\noindent {\small Universit\'e de Bourgogne, 21004 Dijon, FRANCE}

\noindent {\footnotesize{E-mail : bonatti@u-bourgogne.fr}}

\vskip 2mm

\noindent Bin Yu

\noindent {\small Department of Mathematics}

\noindent{\small Tongji University, Shanghai 2000 92, CHINA}

\noindent{\footnotesize{E-mail: binyu1980@gmail.com }}

\end{document}